\documentclass[12pt]{article}
\usepackage[utf8]{inputenc}
\usepackage[T1]{fontenc}
\usepackage{authblk}
\usepackage{amsmath, amssymb, amsthm} 
\usepackage{geometry}         
\usepackage{xcolor}
\usepackage[colorlinks=true, linkcolor=blue, urlcolor=blue, citecolor=blue]{hyperref}       
\usepackage{lipsum}
\usepackage{graphicx}
\newcommand{\supp}{\textnormal{supp }} 
\newcommand{\Uloc}{\textnormal{uloc}}
\newcommand{\Ro}{\mathbb{R}}
\newcommand{\defeq}{\stackrel{\rm def}{=}}

\newcommand{\bslink}[1]{%
   \textsuperscript{\hyperref[page.1]{\textcolor{blue}{\small #1}}}%
}
\providecommand{\keywords}[1]{\noindent\textbf{Keywords:} #1\par\vspace{1em}}
\newtheorem{theorem}{Theorem}[section]
\newtheorem{lemma}[theorem]{Lemma}
\newtheorem{proposition}[theorem]{Proposition}

\newtheorem{definition}[theorem]{Definition}
\newtheorem{remark}[theorem]{Remark}

\numberwithin{equation}{section}
\title{Stability Analysis of Bounded Solutions to 2-D Incompressible 
Navier-Stokes Equations in $L^2_{\Uloc}$ Spaces}
\author[$\bslink{ a}$]{Haokun Tai}
\author[$\bslink{ a}$, $\bslink{$\ast$}$]{Xiaoxin Zheng}
\affil[a ]{School of Mathematical Sciences, Beihang University, Beijing 100191, PR China}
\date{}

\begin{document}

\maketitle

\begin{abstract}
  This paper examines the stability analysis of two-dimensional 
  incompressible fluids in $L^2_{\Uloc}$ spaces. 
  Extending the classical stability framework for Euler equations 
  in Yudovich-type settings, we establish the convergence of 
  Navier-Stokes solutions toward Euler solutions as 
  viscosity $\nu \to 0$, under the weakened assumptions 
  $u^\nu_0 \in L^\infty$ and $\omega^\nu_0 \in L^\infty$
  without imposing decay conditions at infinity.
  We first derive a priori estimates for the Navier-Stokes equations in $L^\infty$ spaces.
  Then, using harmonic analysis tools including Littlewood-Paley
  decomposition and Bony's paraproduct calculus,
  we prove the $L^\infty$-convergence of the velocity field 
  and $L^2_{\Uloc}$-convergence of the vorticity as $\nu \to 0$. 
\end{abstract}



\section{Introduction} \label{sec:intro}
We study the stability of two-dimensional
incompressible Navier-Stokes equations to the Cauthy problem
\begin{equation}\label{eq:n-s}
  \left.\begin{array}{r}
    \partial_t u^\nu + u^\nu \cdot \nabla u^\nu + 
    \nabla p^\nu - \nu \Delta u^\nu = g^\nu  \\
    \text{div } u^\nu = 0 \; \\
    u^\nu \big|_{t=0} = u_0^\nu
  \end{array}\right\},
\end{equation}

in the limit $\nu \to 0$, where $u^\nu =u^\nu(t,x)$ are velocity fields
and $g^\nu = g^\nu(t,x)$ are prescribed source terms,
with $(t,x) \in \Ro^{+} \times \Ro^2$, taking values in $\Ro^2$.
$\nu > 0$ are viscosity coefficient of the fluid.
The pressures $p^\nu = p^\nu(t,x)$ take values in $\Ro^1$.
By means of incompressible conditon $\text{div } u^\nu = 0$,
they also admit the following representation:

\noindent\begin{minipage}{\textwidth}
    \vspace{0.5em}
    \rule{\textwidth}{0.4pt}\\[0.2em]
    {\small $^{\ast}$ Corresponding author.}\\[0.2em]
    {\small \phantom{$^{\ast}$} \textit{E-mail addresses:} \textcolor{blue}{\texttt{haokuntai@buaa.edu.cn}} (H. Tai), 
    \textcolor{blue}{\texttt{xiaoxinzheng@buaa.edu.cn}} (X. Zheng)}
    \label{page.1}
\end{minipage}
\vspace{0.5em}

\begin{align*}
  p^\nu = R_i R_j (u^\nu_i u^\nu_j),
\end{align*}
where $R_i$ is the Riesz transform(see \cite{ref0}).

And the vorticity fields $\omega^\nu(t,x) = \nabla^\perp \cdot u^\nu$ satisfy the
following equation:
\begin{equation}\label{eq:n-s_o}
  \left.\begin{array}{r}
    \partial_t \omega^\nu + u^\nu \cdot \nabla \omega^\nu 
    - \nu \Delta \omega^\nu = \nabla \times g^\nu \\
    u^\nu = K * \omega^\nu\\
    \omega^\nu \big|_{t=0} = \omega_0^\nu
  \end{array}\right\},
\end{equation}
where the operator $K=\frac{1}{2\pi} \frac{x^\perp}{|x|^2}$ from
Biot-Savart law (see \cite{ref2}) and $x^\perp = (-x_2, x_1)$ denotes 
the perpendicular vector to $x = (x_1, x_2)$.

When the inputs $(u^\nu_0 , g^\nu)$ converge to some states $(u_0 , g)$, as $\nu$ vanishes,
it is expected that $u^\nu$ approaches $u$, the solution of the same Euler equations:
\begin{equation}\label{eq:euler}
  \left.\begin{array}{r}
    \partial_t u + u \cdot \nabla u + \nabla p = g \\
    \text{div } u = 0  \\
    u \big|_{t=0} = u_0
  \end{array}\right\}.
\end{equation}

If $(u_0 , g)$ denotes a weak limit of the family $(u^\nu_0 , g^\nu)_{\nu > 0}$,
then $u^\nu$ converges weakly to the unique solution $u$ 
of the Euler equations \eqref{eq:euler} in the sense of distributions.
The corresponding vorticity equations read as follows:
\begin{equation}\label{eq:euler_o}
  \left.\begin{array}{r}
    \partial_t \omega + u \cdot \nabla \omega = \nabla \times g \\
    u = K * \omega\\
    \omega \big|_{t=0} = \omega_0
  \end{array}\right\}.
\end{equation}

The inviscid limit problem has significant practical implications, 
as many real-world fluid flows involve very small viscosity values 
where the inviscid approximation becomes relevant.
In 1983, Kato published a landmark 
contribution \cite{ref:il0} in this direction. 
Several other notable works \cite{ref:il1, ref:il2, ref:il3, ref:il4} have 
addressed the two-dimensional case of these equations.
In addition to the viscosity coefficient,
there are also perturbation problems of solutions
caused by initial data and source terms.

In a recent contribution by Arsénio and Houamed \cite{ref1},
by the idea of compactness extrapolation,
it was demonstrated that given initial data $u^\nu_0 \in H^1$ 
and $\omega^\nu_0 \in L^\infty$, solutions of the two-dimensional 
incompressible Navier-Stokes equations converge 
in $H^1$ to the corresponding solutions of the Euler equations 
as $\nu \to 0$. 

The properties of solutions to the equation in the Yudovich class 
were first established by Yudovich in 1963 \cite{ref:y1}.
For a comprehensive overview of Yudovich class solutions,
we refer to \cite{ref4,ref11}.
Serfati \cite{ref9} proved the existence and uniqueness of 
solutions to two-dimensional Euler equations under the assumption 
that $u_0 \in L^\infty$ and $\omega_0 \in L^\infty$.
Then, Taniuchi (2004) and Bernicot (2014) extended Serfati's results \cite{ref:S2,ref:S4}.
Chen et al. (2019) discussed the results on bmo-type spaces \cite{ref:S5}.
For Navier-Stokes equations, Giga et al. (1999) 
studied global existence in the two-dimensional case\cite{ref:S1}.
Lopes et al. (2021) proved the strong convergence of the vorticities
in the vanishing viscosity limit with unbounded vorticity\cite{ref:S6}.

For stating our main results, we need to introduce the uniformly locally $L^2$ spaces:
\begin{definition}[Uniformly Locally $L^2$ Norm] \label{def:uloc_l2}
  Let $f: \Ro^2 \to \Ro$ be a measurable function. 
  The uniformly locally $L^2$ norm of $f$, 
  denoted by $\|f\|_{L^2_{\Uloc}}$, is defined as
  \[\|f\|_{L^2_{\Uloc}} = \sup_{y \in \Ro^2} 
  \left( \int_{B(y,1)} |f(x)|^2 \, dx \right)^{1/2},\] \\
  where $B(y,1)$ denotes the open ball of radius 1 centered at $y \in \Ro^2$. 
\end{definition}

This leads to our conclusions concerning Log-Lipschitz flows.
Under the weaker assumption 
that $u^\nu_0 \in L^\infty$ and $\omega^\nu_0 \in L^\infty$,
we establish two key convergence results: 
the $L^\infty$-convergence of the velocity field and 
the $L^2_{\Uloc}$-convergence of the vorticity.
While conventional results for Navier-Stokes equations
typically assume initial data in $L^2$ spaces,
our work addresses solutions under fundamentally different conditions.
We remove the restrictive decay requirement at infinity,
a departure that introduces technical challenges.

By Leray projection operator and Littlewood-Paley decomposition
(see Section \ref{sec:pre}), 
we derive a priori estimates for the Navier–Stokes equation with source terms:
\begin{theorem}\label{thm:n-s_control}
  Let $\nu \in (0,1)$ and $T \in \mathbb{R}^{+} \cup \{ \infty \}$. 
  Suppose we are given initial data satisfying
  the following conditions uniformly in $\nu$:
  \begin{align*}
    u^\nu_0 \in L^\infty(\mathbb{R}^2), \quad \omega^\nu_0 \in L^\infty(\mathbb{R}^2).
  \end{align*}
  
  The source terms $g^\nu \in L^{1}([0,T] ;L^\infty(\mathbb{R}^2))$ satisfy 
  $\mathcal{P} g^\nu \in L^{1}([0,T] ;L^\infty(\mathbb{R}^2))$
  and $\nabla \times g^\nu \in L^{1}([0,T];L^\infty(\mathbb{R}^2))$.

  Then, there exists a unique solution $(u^\nu, \omega^\nu)$ 
  to the system \eqref{eq:n-s} and \eqref{eq:n-s_o} such that 
  for all $t \in [0,T]$, the following estimates hold:
  \begin{align*}
    \| u^\nu(t) \|_{L^\infty} & \leq 
    \big(\| u^\nu_0 \|_{L^\infty}  + \| \mathcal{P} g^\nu \|_{L^1([0,t];L^\infty)}\big)
    e^{C t},\\
    \|\omega^\nu(t)\|_{L^\infty} & \leq \|\omega_0^\nu\|_{L^\infty}
    + \| \nabla \times g^\nu \|_{L^1([0,t];L^\infty)},
  \end{align*}
  where $C > 0$ is a constant independent of $\nu$ and $T$.
\end{theorem}

For the Euler equations, by \cite{ref9}, 
we can obtain results of similar form with source terms:
\begin{theorem}\label{thm:euler_control}
  Let $T \in \mathbb{R}^{+} \cup \{ \infty \}$ and
  $(u, \omega)$ be a solution to this system 
  with any initial data and source satisfying
  \begin{align*}
    u_0 \in L^\infty(\Ro^2), \quad \omega_0 \in L^\infty(\Ro^2).
  \end{align*}
  
  The source terms $g \in L^{1}([0,T] ;L^\infty(\mathbb{R}^2))$ satisfy 
  $\mathcal{P} g \in L^{1}([0,T] ;L^\infty(\mathbb{R}^2))$
  and $\nabla \times g \in L^{1}([0,T];L^\infty(\mathbb{R}^2))$.

  Then, for all $t \in [0,T]$, the following estimates hold:
  \begin{align*}
    \| u(t) \|_{L^\infty} & \leq \big(\| u_0 \|_{L^\infty} 
    + \| \mathcal{P} g \|_{L^1([0,t];L^\infty)}\big) 
    e^{C t},\\
    \|\omega(t)\|_{L^\infty} & \leq \|\omega_0\|_{L^\infty} + \| \nabla \times g \|_{L^1([0,t];L^\infty)}.
  \end{align*}
\end{theorem}

Since, we establish the convergence results for solutions 
of the Navier-Stokes equations to those of the Euler equations 
as the viscosity parameter $\nu$ tends to zero. 
Our main findings are presented in the following theorems:

\begin{theorem}\label{thm:convergence}
  Let $T \in \mathbb{R}^+$.
  Consider a family of initial data $(u^\nu_0, g^\nu)_{\nu > 0}$ 
  satisfying the assumptions of Theorems \ref{thm:n-s_control} 
  and \ref{thm:euler_control}, and converging to $ (u_0, g) $ 
  strongly in $ L^\infty(\Ro^2) \times L^1([0,T];L^\infty(\Ro^2))$ as $ \nu \to 0 $.
  Then, it holds that:
  \[
    \lim_{\nu \to 0} \sup_{t \in [0,T]} \| u^\nu(t) - u(t) \|_{L^\infty} = 0,
  \]
  where $ u^\nu $ and $ u $ are the unique solutions 
  of \eqref{eq:n-s} and \eqref{eq:euler}, respectively.
\end{theorem}

For the vorticity field, we have:
\begin{theorem}\label{thm:convergence_omega}
  Let $T \in \mathbb{R}^+$.
  Consider a family of initial data $( \omega^\nu_0, \nabla \times g^\nu)_{\nu > 0}$ 
  satisfying the assumptions of Theorems \ref{thm:n-s_control} 
  and \ref{thm:euler_control}, and converging to $ ( \omega_0, \nabla \times g) $ 
  strongly in $ L^\infty(\Ro^2) \times L^1([0,T];L^\infty(\Ro^2))$ as $ \nu \to 0 $.
  Then, it holds that:
  \[
    \lim_{\nu \to 0} \sup_{t \in [0,T]} \| \omega^\nu(t) - \omega(t) \|_{L^2_{\Uloc}} = 0,
  \]
  where $ \omega^\nu $ and $ \omega $ are the unique solutions 
  of \eqref{eq:n-s_o} and \eqref{eq:euler_o}, respectively.
\end{theorem}

\textbf{Notation.} We will use the following notations.
Let $L^p$ denote the Lebesgue space.
$C_c^\infty$ stands for the space of smooth functions with compact support.
We denote by $\mathcal{S}'_h$ the space of homogeneous tempered distributions.
The symbol $\mathbb{N}^*$ stands for $\mathbb{N} \cup \{-1\}$.
We write $\mathcal{P}$ for the Leray projection operator.
And $G_t(x)$ denotes the two-dimensional heat kernel
corresponding to the operator $ \Delta$, given explicitly by
\[
G_t(x) = \frac{1}{4\pi t} \exp\left(\frac{-|x|^2}{4 t}\right)
\]
for $t > 0$ and $x \in \mathbb{R}^2$.

The paper is organized as follows.
In the next section,
we recall some basic analysis tools and useful properties, 
and introduce several lemmas that will
be essential for our main results.
In Section \ref{sec:est}, 
we establish the estimates for the equations
with the vector field belonging to $L^\infty$.
Section \ref{sec:main} is dedicated to the proof of the main results.

\section{Preliminaries}\label{sec:pre}

\subsection{Littlewood-Paley Theory}
In this subsection, we review the Littlewood-Paley decomposition 
and Bony's paraproduct calculus described in \cite{ref2,ref7,ref12}.
The spaces discussed in this paper are all $\Ro^2$ unless otherwise specified.

\begin{definition}[Inhomogeneous Decomposition]\label{def:inhomogeneous_L-P}
  Let $\chi \in C_c^\infty$ and 
  $\psi  \in C_c^\infty(\Ro^2\setminus\{0\})$ are smooth radial functions that satisfy:
  \begin{itemize}
    \item $\supp \chi \subset B(0,\frac{4}{3})$, $\supp
    \psi \subset \{\xi : \frac{3}{4} \leq |\xi| \leq \frac{8}{3}\}$.
    \item $\chi(\xi) + \sum_{j \geq 0} \psi(2^{-j}\xi) = 1$.
  \end{itemize}

  The corresponding low-frequency cut-off operators $P_{<j}$ 
  and dyadic blocks $P_j$ satisfy:
  \[
    P_{<j} f= C \check{\chi}(2^{j} x)*f,\quad
    P_j f = C 2^{2j} \check{\psi} (2^{j} x)*f, \quad j \geq 0.
  \]
\end{definition}

So we have $f= \sum_{j \geq -1} P_j f$ in $\mathcal{S}'$.
And by the quasi-orthogonality, it is clear that:
\begin{proposition}\label{prop:lp-orthogonality}
  The inhomogeneous Littlewood-Paley decomposition satisfies the following properties:
  \begin{enumerate}
    \item[(i)] For any $j, k \geq 0$ with $|j - k| \geq 2$, we have
    \begin{align*}
      P_j P_k f = 0, \quad \forall f \in \mathcal{S}'.
    \end{align*}
    
    \item[(ii)] For any $j, k \geq 0$ with $|j - k| \geq 5$, we have
    \begin{align*}
      P_j (P_{<k-1} f P_k g) = 0, \quad \forall f,g \in \mathcal{S}'.
    \end{align*}
  \end{enumerate}
\end{proposition}

Now we define the homogeneous Littlewood-Paley decomposition:
\begin{definition}[Homogeneous Decomposition] \label{def:homogeneous_L-P}
Let $\psi \in C_c^\infty(\Ro^2 \setminus\{0\})$ be a smooth radial function that satisfies:
  \begin{itemize}
    \item $\supp \psi \subset \{\xi : \frac{3}{4} \leq |\xi| \leq \frac{8}{3}\}$
    \item $\sum_{j\in\mathbb{Z}} \psi(2^{-j}\xi) = 1,\quad \forall \xi\in\Ro^2\setminus\{0\}$
  \end{itemize}

The corresponding homogeneous dyadic blocks $\dot{P}_j$ satisfy:
\[
  \dot{P}_j f = C 2^{2j} \check{\psi}(2^{j} x)*f, \quad j\in\mathbb{Z}.
\]

And the homogeneous low-frequency cut-off operator $\dot{P}_j$ is defined by:
\[
  \dot{P}_j f = \sum_{k<j} \dot{P}_k f.
\]
\end{definition}

We often use the notation $P_{\geq n} f$ to denote the other part of $f$:
\[
  P_{\geq n} f \defeq (I - P_{<n}) f.
\]

\begin{remark} \label{rem:homogeneous-lp-properties}
  \begin{enumerate}
    \item[(i)] For any $u \in \mathcal{S}'$, 
    $u \in \mathcal{S}'_h$ if and only if:
    \begin{align*}
      \lim_{n \to -\infty} \dot{P}_{<n} u = 0,\quad \in \mathcal{S}'.
    \end{align*}
    
    \item[(ii)] In the sense of distributions, 
    for any $u \in \mathcal{S}'_h$, we have:
    \begin{align*}
      \sum_{k<j} \dot{P}_k u = \dot{P}_{<j} u.
    \end{align*}
  \end{enumerate}
\end{remark}

Then, we review Bernstein inequalities, which are useful tools for the proof.
\begin{lemma}[Bernstein's Inequalities] \label{lem:bernstein}
  Let $f \in \mathcal{S}'_h$ be a homogeneous tempered 
  distribution. Then the following inequalities hold for the 
  homogeneous dyadic blocks $\dot{P}_j$ and the homogeneous 
  low-frequency cut-off operator $\dot{P}_j$:
  
  \begin{enumerate}
    \item[(i)] Basic Bernstein inequality: For any $1 \leq p \leq q \leq \infty$ and any $j \in \mathbb{Z}$, there exists a constant $C > 0$ such that
    \begin{align*}
      \|\dot{P}_j f\|_{L^q} \leq 
      C 2^{2 j (\frac{1}{p} - \frac{1}{q})} \|\dot{P}_j f\|_{L^p},
    \end{align*}
    
    \item[(ii)] Derivative Bernstein inequality: For any $1 \leq p \leq \infty$, any $j \in \mathbb{Z}$, and any multi-index $\alpha \in \mathbb{N}_0^d$, there exists a constant $C_\alpha > 0$ such that
    \begin{align*}
      \|\partial^\alpha \dot{P}_j f\|_{L^p} &\leq C_\alpha 2^{j |\alpha|} \|\dot{P}_j f\|_{L^p}, \\
      \|\partial^\alpha \dot{P}_{<j} f\|_{L^p} &\leq C_\alpha 2^{j |\alpha|} \|\dot{P}_{<j} f\|_{L^p}.
    \end{align*}
    
    \item[(iii)] Spectral localization: For any $1 \leq p \leq \infty$ and any $j \in \mathbb{Z}$, there exists a constant $C > 0$ such that
    \begin{align*}
      \|\dot{P}_j f\|_{L^p} &\approx 2^{-j} \|\nabla \dot{P}_j f\|_{L^p}.
    \end{align*}
  \end{enumerate}
\end{lemma}

And we also have some functional spaces from Littlewood-Paley decomposition.
\begin{definition}[Homogeneous Besov Spaces] \label{def:homogeneous-besov}
  Let $s \in \Ro$ and $1 \leq p, q \leq \infty$. The homogeneous Besov space $\dot{B}^s_{p,q}$ is defined as the space of all homogeneous tempered distributions $f \in \mathcal{S}'_h$ such that
  \begin{align*}
    \|f\|_{\dot{B}^s_{p,q}} \defeq 
    \| 2^{js} \|\dot{P}_j f\|_{L^p} \|_{\ell^q} < \infty,
  \end{align*}
  
  The space $\dot{B}^s_{p,q}$ is a Banach space when equipped with the above norm.
\end{definition}

\begin{definition}[Inhomogeneous Besov Spaces] \label{def:inhomogeneous-besov}
  Let $s \in \Ro$ and $1 \leq p, q \leq \infty$. The inhomogeneous Besov space $B^s_{p,q}$ is defined as the space of all tempered distributions $f \in \mathcal{S}'$ such that
  \begin{align*}
    \|f\|_{B^s_{p,q}} \defeq 
    \| 2^{js} \|P_j f\|_{L^p} \|_{\ell^q(\mathbb{N}^*)} < \infty,
  \end{align*}
  where $\Delta_j$ are the inhomogeneous dyadic blocks for $j \geq 0$,
  and $\Delta_{-1} = P_0$ is the low-frequency cut-off operator.
\end{definition}

\begin{definition}[Homogeneous Triebel-Lizorkin Spaces] \label{def:homogeneous-triebel}
  Let $s \in \Ro$, $1 \leq p < \infty$, and $1 \leq q \leq \infty$.
  The homogeneous Triebel-Lizorkin space $\dot{F}^s_{p,q}$ is defined as the space of all homogeneous tempered distributions $f \in \mathcal{S}'_h$ such that
  \begin{align*}
    \|f\|_{\dot{F}^s_{p,q}} \defeq 
    \left\|  \| 2^{js} |\dot{P}_j f| \|_{\ell^q}  \right\|_{L^p} < \infty,
  \end{align*}
\end{definition}

According to our initial data, the BMO space is important for the main theorem.
\begin{definition}[BMO Space] \label{def:bmo-space}
  The space of functions of bounded mean oscillation \\
  $\mathrm{BMO}$ is defined as the 
  space of all locally integrable functions 
  $f \in L^1_{\mathrm{loc}}$ such that
  \begin{align*}
    \|f\|_{\mathrm{BMO}} \defeq \sup_{B} \frac{1}{|B|} \int_B |f(x) - f_B| \, dx < \infty,
  \end{align*}
  where the supremum is taken over all balls $B \subset \Ro^2$ and $f_B = \frac{1}{|B|} \int_B f(y) \, dy$.
\end{definition}
Since $\| \cdot \|_{\mathrm{BMO}}$ is a seminorm, 
we treat functions differing by a constant as the same equivalence class.

Finally, we introduce Bony decomposition,
which is critical for the product term in the equation.
\begin{definition}[Bony Decomposition] \label{def:bony-decomposition}
  Let $f, g \in \mathcal{S}'_h$. 
  The product $fg$ can be decomposed into three parts:
  \begin{align*}
    fg = \dot{T}_f g + \dot{T}_g f + \dot{R}(f,g),
  \end{align*}
  where the paraproduct terms $T_f g$, $T_g f$ and the remainder term $R(f,g)$ are defined as follows:
  
  \begin{enumerate}
    \item[(i)] Homogeneous paraproduct terms:
    \begin{align*}
      \dot{T}_f g & \defeq \sum_{j\in\mathbb{Z}} \dot{P}_{<j-1} f \, \dot{P}_j g = \sum_{j\in\mathbb{Z}} \sum_{k<j-1} \dot{P}_k f \, \dot{P}_j g,
    \end{align*}
    
    \item[(ii)] Homogeneous remainder term:
    \begin{align*}
      \dot{R}(f,g) \defeq \sum_{j\in\mathbb{Z}} \sum_{|j-k|\leq 1} \dot{P}_j f \, \dot{P}_k g
      = \sum_{j\in\mathbb{Z}} \dot{P}_j f \, \widetilde{\dot{P}}_j g,
    \end{align*}
  \end{enumerate}
  where $\widetilde{\dot{P}}_j \defeq \sum_{k=-1}^{1} \dot{P}_{j-k}$.

  Here $\dot{P}_{<j}$ denotes the homogeneous low-frequency cut-off operator
  and $\dot{P}_j$ denotes the homogeneous dyadic block operator from the homogeneous Littlewood-Paley decomposition (see Definition \ref{def:homogeneous_L-P}).
\end{definition}

We omit the inhomogeneous decomposition,
which has a similar structure.

In the next subsection, we also gather some critical properties about BMO 
and homogeneous Triebel–Lizorkin spaces.

\subsection{Some Useful Properties and Lemmas}
In this subsection, we collect some useful properties and lemmas 
about uniformly locally $L^2$ spaces and other functional spaces.

We employ a smooth cut-off function $\varphi_{R,y}(x)$ to simplify the expression:
\begin{definition}\label{def:cut-off}
  A function $\varphi: \Ro^2 \to [0, 1]$ is called a smooth cut-off function if it satisfies:
  \[
  \varphi(x) = 
  \begin{cases}
  1, & \text{for } x \in B(0, 1), \\
  \text{smooth transition}, & \text{for } x \in B(0, 2) \setminus B(0, 1), \\
  0, & \text{for } x \in \Ro^2 \setminus B(0, 2),
  \end{cases}
  \]
  where $B(0, r)$ denotes the open ball of radius $r$ centered at the origin, and $\varphi \in C_0^\infty$.
\end{definition}

\begin{proposition}\label{prop:cut-off}
  Let $\varphi_{R,y}(x) = \varphi(\frac{x-y}{R})$ where $\varphi$ is the smooth cut-off function defined in Definition \ref{def:cut-off}. Then:
  \begin{enumerate}
    \item[(i)] The uniformly locally $L^2$ norm admits the following equivalent form for any $R > 0$:
    \begin{align*}
      \|f\|_{L^2_{\Uloc}} \simeq \sup_{y} \|\varphi_{R,y} f\|_{L^2}.
    \end{align*}

    \item[(ii)] There exists a constant $C > 0$ such that for any $f \in L^p (p \geq 2)$,
    \begin{align*}
      \|f\|_{L^2_{\Uloc}} \leq C\|f\|_{L^p}.
    \end{align*}
  \end{enumerate}
\end{proposition}

This space serves, in a certain sense,
as an alternative to the $L^\infty$ space.
More results and propositions of uniformly
locally $L^2$ solutions can be found in \cite{ref5}.

For the $L^\infty$-convergence of the velocity field,
we also need the sharp embedding inequality:
\begin{proposition}\label{prop:inequality_uloc_l2}
  Let $u$ be a solution to \eqref{eq:n-s} or \eqref{eq:euler} and 
  $\omega$ to \eqref{eq:n-s_o} or \eqref{eq:euler_o} with 
  initial data $u_0 \in L^\infty$ and
  $\omega_0 \in L^\infty$.
  Then, for any $t > 0$, the solution satisfies:
  \begin{align*}
    \| u \|_{L^\infty} \leq \| u \|^{\frac{1}{2}}_{L^2_{\Uloc}} 
    \| \omega \|^{\frac{1}{2}}_{B^0_{\infty,\infty}}.
  \end{align*}
\end{proposition}

\begin{proof}[Proof of Proposition \ref{prop:inequality_uloc_l2}]
Applying the Littlewood-Paley theory, we have:
\begin{align*}
  \| u \|_{L^\infty} \leq \| P_{\geq n}  u \|_{L^\infty} + \|P_{<n}  u \|_{L^\infty},
\end{align*}
where $N>1$ to be fixed later.

For the high-frequency part, we control its norm through the vorticity.
So for $P_{\geq N}  u$, it implies for each $N \geq 0$ that:
\begin{align*}
  \| P_{\geq N} u \|_{L^\infty} &\leq \sum_{q \geq N} \| P_q  u \|_{L^\infty} \\
  &\leq \sum_{q \geq N} C 2^{-q}\| P_q \nabla  u \|_{L^\infty} \\
  &\leq \sum_{q \geq N} C 2^{-q}\| P_q \ \omega \|_{L^\infty} \\
  &\leq \sup_{q \geq -1} \| P_q \ \omega \|_{L^\infty} \sum_{q \geq N} C 2^{-q} \\
  &\leq C 2^{-N} \| \omega \|_{B^0_{\infty,\infty}}.
\end{align*}

According to the proposition of operator $P_{<n}$ 
and $L^2_{\Uloc}$ space, we have:
\begin{align*}
  \|P_{<N} u \|_{L^\infty} \leq C 2^N\| u \|_{L^2_{\Uloc}}.
\end{align*}

By exploiting the arbitrariness of the coefficient $N$ and let 
$2^N = C \| u \|^{-\frac{1}{2}}_{L^2_{\Uloc}} \| \omega \|^{\frac{1}{2}}_{B^0_{\infty,\infty}}$, we conclude that:
\begin{align*}
  \|  u \|_{L^\infty} \leq C' \|  u \|^{\frac{1}{2}}_{L^2_{\Uloc}} 
  \| \omega \|^{\frac{1}{2}}_{B^0_{\infty,\infty}}.
\end{align*}
\end{proof}

And we directly have the following inequality:
\begin{align*}
  \| u \|_{L^\infty} \leq C \| u \|^{\frac{1}{2}}_{L^2_{\Uloc}} 
  \| \omega \|^{\frac{1}{2}}_{L^\infty}.
\end{align*}

The convolution in $L^2_{\Uloc}$ also have similar conclusion(see \cite{ref5}):
\begin{proposition}\label{prop:uloc_convolution}
  For any $f \in L^1$ and $g \in L^2_{\Uloc}$,
  their convolution $f * g$ belongs to $L^2_{\Uloc}$,
  and there exists a universal constant $C > 0$ such that
  \begin{align*}
    \| f * g \|_{L^2_{\Uloc}} \leq C \| f \|_{L^1} \| g \|_{L^2_{\Uloc}}.
  \end{align*}
\end{proposition}

In view of the norm structure, it is important in the 
proof of the vorticity that(see \cite{ref8}):
\begin{remark}\label{rem:bmo-fsp-equivalence}
  Let $f \in \mathcal{S}'$ be a tempered distribution. 
  Then the following norm equivalence holds:
  \begin{align*}
    \| f \|_{\mathrm{BMO}} \approx \| f \|_{\dot F^0_{\infty,2}},
  \end{align*}
\end{remark}

And it is also important for the boundedness that(see \cite{ref0} and \cite{ref6}):
\begin{remark}\label{rem:riesz-bmo-properties}
  Recall that the $j$-th Riesz transform $R_j$ on $\Ro^2$ is defined by
  \begin{align*}
    R_j f(x) = C \int_{\Ro^2} \frac{x_j - y_j}{|x-y|^{2}} f(y) \, dy,
    \quad j = 1,2.
  \end{align*}
  where $C$ is a constant.

  And we have $R_j$ maps $\mathrm{BMO}$ to $\mathrm{BMO}$ boundedly.
\end{remark}

\section{A priori estimates}\label{sec:est}
In this section, we present the a priori estimates 
for the vorticity $u^\nu$ and the velocity field $\omega^\nu$.
We mainly use the heat kernel $G_t(x)$ and the Bony decomposition.
Through the Theorem \ref{thm:n-s_control}, 
we can observe that the $L^\infty$ estimates for the N-S equations
are independent of the viscosity coefficient $\nu$.

By applying the maximum principle(see \cite{ref3}), we have:
\begin{align*}
  \| \omega^{\nu}(t) \|_{L^\infty} \leq \| \omega^{\nu}_0 \|_{L^\infty}
  + \| \nabla \times g^\nu \|_{L^1([0,t];L^\infty)}.
\end{align*}

For the velocity field and equation \eqref{eq:n-s}, 
using heat kernel $G_t(x)$ and Leray projection operator $\mathcal{P}$, we can write:
\begin{align}\label{eq:n-s_heat}
  u^\nu(t) = G_{\nu t}(x) * u^\nu_0
  - \int^{t}_{0} G_{\nu(t-s)}(x) * \mathcal{P} ( u^\nu \cdot \nabla u^\nu ) \,ds
  + \int^{t}_{0} G_{\nu(t-s)}(x) * \mathcal{P} g^\nu \,ds.
\end{align}

According to triangular inequality and Young inequality, we have:
\begin{align*}
  & \| u^\nu \|_{L^\infty}\\
  \leq & \|G_{\nu t}(x) * u^\nu_0 \|_{L^\infty} 
  + \int^t_0 \| G_{\nu(t-s)}(x) * \mathcal{P} ( u^\nu \cdot \nabla u^\nu ) \|_{L^\infty} ds
  + \int^t_0 \| G_{\nu(t-s)}(x) * \mathcal{P} g^\nu \|_{L^\infty} ds\\
  \leq & \| u^\nu_0 \|_{L^\infty}
  + \int^t_0 \| G_{\nu(t-s)}(x)  \|_{L^1} \| \mathcal{P} ( u^\nu \cdot \nabla u^\nu ) \|_{L^\infty} ds
  + \int^t_0 \| G_{\nu(t-s)}(x) \|_{L^1} \| \mathcal{P} g^\nu \|_{L^\infty} ds\\    
  \leq & \| u^\nu_0 \|_{L^\infty} + \int^t_0 \| \mathcal{P} ( u^\nu \cdot \nabla u^\nu ) \|_{L^\infty} ds
  + \int^t_0 \| \mathcal{P} g^\nu \|_{L^\infty} ds.
\end{align*}

Then, we use the pressure's two forms.
On the one hand, by operator $\mathcal{P}$, we have:
\begin{align*}
  \mathcal{P} ( u^\nu \cdot \nabla u^\nu ) = \nabla p^\nu.
\end{align*}

On the other hand(see Section \ref{sec:intro}):
\begin{align*}
  p^\nu = \frac{1}{-\Delta} (\partial_i u^\nu_j \partial_j u^\nu_i).
\end{align*}

So we need to show the control of $\| \mathcal{P} (u^\nu \cdot \nabla u^\nu)\|_{L^\infty}$
or $\| \nabla p^\nu \|_{L^\infty}$.
First, we discuss the low frequencies of $\| \mathcal{P} (u^\nu \cdot \nabla u^\nu) \|_{L^\infty}$
by Littlewood-Paley theory.
Then, we control the high frequencies of $\| \nabla p^\nu \|_{L^\infty}$ and 
$\| \mathcal{P} (u^\nu \cdot \nabla u^\nu)\|_{L^\infty}$ 
by Bony's paraproduct decomposition.

By the boundedness of singular integral operator $\mathcal{P}$:
\begin{align*}
  \sum^{N-1}_{q = - \infty} \|P_q \mathcal{P} (u^\nu \cdot \nabla u^\nu)\|_{L^\infty} 
  & \leq \sum^{N-1}_{q = - \infty} C \|P_q (u^\nu \cdot \nabla u^\nu)\|_{L^\infty} \\
  & \leq \sum^{N-1}_{q = - \infty} C \|P_q \nabla \cdot (u^\nu \otimes u^\nu)\|_{L^\infty} \\
  & \leq \sum^{N-1}_{q = - \infty} C 2^q\| (u^\nu \otimes u^\nu)\|_{L^\infty} \\
  & \leq C 2^N \| u^\nu \|^2_{L^\infty}.
\end{align*}

For pressure $p^\nu$ and $u^\nu \otimes u^\nu$($=(u^\nu_i u^\nu_j)_{2 \times 2}$),
we have the following Bony decomposition:
\begin{align*}
  \nabla p^\nu & = \frac{\partial_k}{-\Delta} (\partial_i u^\nu_j \partial_j u^\nu_i)\\
  & = \frac{\partial_k}{-\Delta} T_{\partial_i u^\nu_j} \partial_j u^\nu_i 
  + \frac{\partial_k}{-\Delta} T_{\partial_j u^\nu_i} \partial_i u^\nu_j
  + \frac{\partial_k}{-\Delta} R(\partial_i u^\nu_j, \partial_j u^\nu_i),
\end{align*}

and
\begin{align*}
  (u^\nu \cdot \nabla u^\nu) & =\partial_i T_{u^\nu_i} u^\nu_j
  + \partial_i T_{u^\nu_j} u^\nu_i + \partial_i R(u^\nu_i, u^\nu_j).
\end{align*}

For the high-frequency case,
we replace the original first term in $u^\nu \cdot \nabla u^\nu$ 
by the first term in $\nabla p^\nu$. 
Thus, it follows that:
\begin{align*}
  & \sum^{\infty}_{q = N} \|P_q \mathcal{P} (u^\nu \cdot \nabla u^\nu)\|_{L^\infty} \\
  \leq & \sum_{q = N}^{+ \infty}  \| P_q \frac{\partial_k}{-\Delta} 
  T_{\partial_i u^\nu_j} \partial_j u^\nu_i\|_{L^\infty}
  + \sum^{ \infty }_{q = N} C \|P_q \partial_i T_{u^\nu_j} u^\nu_i\|_{L^\infty}
  + \sum^{ \infty }_{q = N} C \|P_q \partial_i R(u^\nu_i, u^\nu_j)\|_{L^\infty}.
\end{align*}

For these three parts, we successively conclude that:
\begin{align*}  
  \sum_{q = N}^{+ \infty}  \| P_q \frac{\partial_k}{-\Delta} 
  T_{\partial_i u^\nu_j} \partial_j u^\nu_i\|_{L^\infty}
  & \leq \sum_{q = N}^{+ \infty} 2^{-q} \| P_q T_{\partial_i u^\nu_j} 
  \partial_j u^\nu_i\|_{L^\infty} \\
  & \leq \sum_{q = N}^{+ \infty} 2^{-q} 
  \| P_q \sum_{ |m-q|<5 } P_{<m-1}
  \partial_i u^\nu_j P_m \partial_j u^\nu_i\|_{L^\infty} \\
  & \leq \sum_{q = N}^{+ \infty} 2^{-q} \sum_{ |m-q|<5 }
  \| P_{<m-1}\partial_i u^\nu_j P_m \partial_j u^\nu_i\|_{L^\infty}\\
  & \leq \sum_{q = N}^{+ \infty} 2^{-q} \sum_{ |m-q|<5 } 
  2^{\frac{m}{2}} \sup_{m} 2^{-\frac{m}{2}}
  \| P_{<m-1}\partial_i u^\nu_j \|_{L^\infty} \| P_m \partial_j u^\nu_i\|_{L^\infty} \\
  & \leq \sum_{q = N}^{+ \infty} 2^{-q} \sum_{ |m-q|<5 } 2^{\frac{m}{2}}
  \| \partial_i u^\nu_j \|_{\dot B^{-\frac{1}{2}}_{\infty , \infty}} \|\omega^\nu\|_{L^\infty} \\
  & \leq 2^{-N} 2^{\frac{N}{2}} \| u^\nu \|_{\dot B^{\frac{1}{2}}_{\infty , \infty}} 
  \|\omega^\nu\|_{L^\infty} \\
  & \leq 2^{-\frac{N}{2}} \| u^\nu \|^{\frac{1}{2}}_{ L^{\infty} } 
  \|\omega^\nu\|^{\frac{3}{2}}_{L^\infty},
\end{align*}
where we use the proposition \ref{prop:lp-orthogonality} (ii) and embedding inequality in besov space.

Then, we apply the same method to control the next term:
\begin{align*}
  \sum^{ \infty }_{q = N} \|P_q \partial_i T_{u^\nu_j} u^\nu_i\|_{L^\infty}
  & \leq \sum^{ \infty }_{q = N} \|P_q T_{\partial_i u^\nu_j} u^\nu_i\|_{L^\infty} \\
  & \leq \sum^{ \infty }_{q = N} \sum_{|m-q|<5} 
  \| P_{<m-1} \partial_i u^\nu_j P_m u^\nu_i\|_{L^\infty} \\
  & \leq \sum^{ \infty }_{q = N} \sum_{|m-q|<5} 
  \| P_{<m-1} \partial_i u^\nu_j\|_{L^\infty} \| P_m u^\nu_i\|_{L^\infty} \\
  & \leq \sum^{ \infty }_{q = N} \sum_{|m-q|<5} 
  2^{-m} \| P_{<m-1} \partial_i u^\nu_j\|_{L^\infty} \| P_m \partial_i u^\nu_i\|_{L^\infty} \\
  & \leq \sum^{ \infty }_{q = N} \sum_{|m-q|<5} 2^{-m/2} 
  \| P_{<m-1} \partial_i u^\nu_j\|_{L^\infty} 2^{-m/2}  
  \| P_m \partial_i u^\nu_i\|_{L^\infty} \\
  & \leq 2^{-\frac{N}{2}} \| \partial_i u^\nu_j\|_{\dot B^{-\frac{1}{2}}_{\infty , \infty}} 
  \| \omega^\nu\|_{L^\infty} \\
  & \leq 2^{-\frac{N}{2}}  \| \omega^\nu\|^{\frac{3}{2}}_{L^\infty} 
  \| u^\nu \|^{\frac{1}{2}}_{L^\infty}.
\end{align*}

For the remainder term, by Hölder's inequality, we have:
\begin{align*}
  \sum^{ \infty }_{q = N} \|P_q \partial_i R(u^\nu_i, u^\nu_j)\|_{L^\infty} 
  & \leq \sum^{ \infty }_{q = N} 2^{-q} 2^{q} 
  \|P_q \partial_j R(u^\nu_i, u^\nu_j)\|_{L^\infty} \\
  & \leq \sum^{ \infty }_{q = N} 2^{-q} \sup_{q}
  2^{q} \|P_q \partial_j R(u^\nu_i, u^\nu_j)\|_{L^\infty} \\
  & \leq \sum^{ \infty }_{q = N} 2^{-q} \sup_{q}
  2^{q} \|P_q \partial_j 
  (\sum_m P_m u^\nu_i \widetilde{P}_m u^\nu_j) \|_{L^\infty} \\
  & \leq \sum^{ \infty }_{q = N} 2^{-q} \sup_{q}
  2^{2q} \sum_m \|P_m u^\nu_i\|_{L^\infty} 
  \|\widetilde{P}_m u^\nu_j\|_{L^\infty}.
\end{align*}

Applying the relation between the vorticity and the velocity field
on the localized frequency, we obtain:
\begin{align*}
  \sum^{ \infty }_{q = N} \|P_q \partial_i R(u^\nu_i, u^\nu_j)\|_{L^\infty}
  & \leq \sum^{ \infty }_{q = N} 2^{-q} \sup_{q}
  2^{2q} \sum_{m>q-5} 2^{-2m} \|\omega^\nu\|^2_{L^\infty} \\
  & \leq \sum^{ \infty }_{q = N} 2^{-q} \sup_{q}
  2^{2q} 2^{-2(q-5)} \|\omega^\nu\|^2_{L^\infty} \\
  & \leq 2^{-N} \| \omega^\nu \|^2_{L^\infty}.
\end{align*}

Consequently, by Young inequality and the control of vorticity, we derive:
\begin{align*}
  & \| u^\nu \|_{L^\infty} \\
  \leq & \int^t_0 [C 2^N \| u^\nu \|^2_{L^\infty} +
  C 2^{-\frac{N}{2}} \| u^\nu \|^{\frac{1}{2}}_{L^\infty} 
  \| \omega^\nu \|^{\frac{3}{2}}_{L^\infty} +
  C 2^{-N} \| \omega^\nu \|^2_{L^\infty} ] ds \\
  & + \| u^\nu_0 \|_{L^\infty} 
  + \|\mathcal{P}  g^\nu \|_{L^1([0,t];L^\infty)}\\
  \leq & \| u^\nu_0 \|_{L^\infty} + \int^t_0 [
  C 2^N \| u^\nu \|^2_{L^\infty} + C 2^{-N} \| \omega^\nu \|^2_{L^\infty}] ds
  + \|\mathcal{P}  g^\nu \|_{L^1([0,t];L^\infty)}\\
  \leq & \| u^\nu_0 \|_{L^\infty} + C \| \omega^\nu \|_{L^\infty} 
  \int^t_0 \| u^\nu \|_{L^\infty} ds
  + \|\mathcal{P}  g^\nu \|_{L^1([0,t];L^\infty)},
\end{align*}
where we set $2^N = C' \| u^\nu \|^{-1}_{L^\infty} \| \omega^\nu \|_{L^\infty} $.

Finally, according to Gronwall inequality, we have the result:
\begin{align*}
  \| u^\nu \|_{L^\infty} & \leq 
  (\| u^\nu_0 \|_{L^\infty}+ \|\mathcal{P}  g^\nu \|_{L^1([0,t];L^\infty)})
   e^{C\| \omega^\nu \|_{L^\infty} t}.
\end{align*}

Their form is similar to the estimates for the Euler equations(Theorem \ref{thm:euler_control}).

\section{Proof of the main theorems}\label{sec:main}
This section is dedicated to establishing the convergence results 
presented in the previous section.
We first prove the convergence of the velocity fields (Theorem \ref{thm:convergence}).
For certain terms, we adopt a specific annular decomposition approach.
Then, we establish the convergence of the vorticity fields (Theorem \ref{thm:convergence_omega}).
To address some challenging terms, we resort to the method of approximate equations.

\subsection{Proof of Theorem \ref{thm:convergence}}
In this subsection, we focus on the convergence of the velocity fields.
We analyze directly by subtracting the two equations.
For the functional space we choose, we need to use
the equivalent form of the $L^2_{\Uloc}$ norm in Section \ref{sec:intro}.

Let $\delta u = u^\nu - u$, $\delta p = p^\nu - p$, and $\delta g = g^\nu - g$.
Considering the equations \eqref{eq:n-s} and \eqref{eq:euler} with cut-off
function $\varphi_{R,y}(x)$, we have:
\begin{equation}\label{eq:d_v}
  \begin{array}{rl}
    &\partial_t (\varphi_{R,y} \delta u) + (u^\nu \cdot \nabla) (\varphi_{R,y} \delta u) 
    + (\varphi_{R,y} \delta u) \cdot \nabla u\\
    =& \nu \varphi_{R,y} \Delta u^\nu - \varphi_{R,y} \nabla \delta p
    + (u^\nu \cdot \nabla \varphi_{R,y} ) \delta u + \varphi_{R,y} \delta g.
  \end{array}
\end{equation}

Multiplying by $ \varphi_{R,y} \delta u $ and then integrating the resulting equality,
we thus obtain the following successive conclusions,
for the first term on the right-hand side:
\begin{align*}
  \nu \int_{\Ro^2} \varphi_{R,y} \Delta u^\nu  \varphi_{R,y} \delta u \,dx
  & \leq \nu \|  \varphi_{R,y} \Delta u^\nu \|_{L^2} \| \varphi_{R,y} \delta u \|_{L^2} \\
  & \leq \nu \|  \varphi_{R,y} \nabla \times \omega^\nu \|_{L^2} \| \varphi_{R,y} \delta u \|_{L^2} \\
  & \leq \nu \|  \varphi_{R,y} \nabla \omega^\nu \|_{L^2} \| \varphi_{R,y} \delta u \|_{L^2}.
\end{align*}

Using the equation \eqref{eq:n-s_o} and $ \varphi_{R,y} $, we have:
\begin{align*}
  \partial_t \varphi_{R,y} \omega^\nu + \varphi_{R,y} (u^\nu \cdot \nabla) \omega^\nu
  = & \nu \varphi_{R,y} \Delta \omega^\nu + \varphi_{R,y} \nabla \times g^\nu,
\end{align*}
which employ
\begin{align*}
  & \partial_t \varphi_{R,y} \omega^\nu + (u^\nu \cdot \nabla) (\varphi_{R,y} \omega^\nu) \\
  = & \nu \Delta (\varphi_{R,y} \omega^\nu) - (u \cdot \nabla \varphi_{R,y}) \omega^\nu - \nu \omega^\nu \Delta \varphi_{R,y}
  - 2 \nu (\nabla \omega^\nu) \cdot (\nabla \varphi_{R,y}) + \varphi_{R,y} \nabla \times g^\nu.
\end{align*}

Multiplying by $ \varphi_{R,y} \omega^\nu $ and integrating the resulting equality,
it follows that:
\begin{align*}
  &\frac{1}{2} \frac{d}{dt} \| \varphi_{R,y} \omega^\nu \|^2_{L^2} +
  \nu \| \nabla (\varphi_{R,y} \omega^\nu) \|^2_{L^2}\\
  \leq & \int_{\Ro^2} (u^\nu \cdot \nabla \varphi_{R,y}) \omega^\nu \varphi_{R,y} \omega^\nu \,dx 
  + \nu \int_{\Ro^2} \omega^\nu (\Delta \varphi_{R,y}) \varphi_{R,y} \omega^\nu \,dx
  + \| \varphi_{R,y} \nabla \times g^\nu \|_{L^2} \| \varphi_{R,y} \omega^\nu \|_{L^2} \\
  & + 2 \nu \int_{\Ro^2} (\nabla \omega^\nu) \cdot (\nabla \varphi_{R,y}) \varphi_{R,y} \omega^\nu \,dx \\
  \leq &  \| u^\nu \|_{L^\infty} \| \nabla \varphi_{R,y} \|_{L^\infty} 
  \| \varphi_{2R,y} \omega^\nu \|_{L^2} \| \varphi_{R,y} \omega^\nu \|_{L^2} 
  + 2 \nu \| \varphi_{2R,y} \omega^\nu \|_{L^2} 
  \| \Delta \varphi_{R,y} \|_{L^\infty} \| \varphi_{R,y} \omega^\nu \|_{L^2} \\
  & + \| \varphi_{R,y} \nabla \times g^\nu \|_{L^2} \| \varphi_{R,y} \omega^\nu \|_{L^2}
  + \nu \| \nabla \varphi_{R,y} \|_{L^\infty} 
  \| \varphi_{2R,y} \omega^\nu \|_{L^2} \| \nabla (\varphi_{R,y} \omega^\nu) \|_{L^2}.
\end{align*}

For the last term, by Young inequality, we have:
\begin{align*}
  \nu \| \nabla \varphi_{R,y} \|_{L^\infty} 
  \| \varphi_{2R,y} \omega^\nu \|_{L^2} \| \nabla (\varphi_{R,y} \omega^\nu) \|_{L^2} 
  &\leq \frac{\nu}{2} \| \nabla \varphi_{R,y} \|^2_{L^\infty} \| \varphi_{2R,y} \omega^\nu \|^2_{L^2}
  + \frac{\nu}{2} \| \nabla (\varphi_{R,y} \omega^\nu) \|^2_{L^2}.
\end{align*}

\begin{remark}\label{rem:phi_R_2R_relation}
For the cut-off functions $\varphi_{R,y}$ and $\varphi_{2R,y}$,
we have the following relationship between their $L^2$ norms:
  \begin{align*}
    \|\varphi_{R,y} f\|_{L^2} \leq \|\varphi_{2R,y} f\|_{L^2} 
    \leq C  \sup_{y} \|\varphi_{R,y} f\|_{L^2},
  \end{align*}
for some constant $C > 0$ independent of $R$ and $y$. 
\end{remark}

Applying the proposition \ref{prop:cut-off} and 
taking the supremum of the right-hand side of the inequality over $y \in \Ro^2$,
we find that:
\begin{align*}
  & \frac{1}{2} \frac{d}{dt} \| \varphi_{R,y} \omega^\nu \|^2_{L^2} +
  \frac{1}{2} \nu \| \nabla (\varphi_{R,y} \omega^\nu) \|^2_{L^2} \\
  \leq & (\frac{C_1}{R} \| u^\nu \|_{L^\infty} + \frac{C_2}{R^2} \nu) 
  \sup_{y} \| \varphi_{R,y} \omega^\nu \|^2_{L^2}
  + \sup_{y} \| \varphi_{R,y} \nabla \times g^\nu \|_{L^2}
  \sup_{y} \| \varphi_{R,y} \omega^\nu \|_{L^2}\\
  \leq & (\frac{C_1}{R} e^{C t}
  + \frac{C_2}{R^2} \nu) \sup_{y} \| \varphi_{R,y} \omega^\nu \|^2_{L^2}
  + \sup_{y} \| \varphi_{R,y} \nabla \times g^\nu \|_{L^2}
  \sup_{y} \| \varphi_{R,y} \omega^\nu \|_{L^2} \\
  \leq & (\frac{C_1}{R} e^{C t}
  + \frac{C_2}{R^2} \nu) \sup_{y} \| \varphi_{R,y} \omega^\nu \|^2_{L^2}
  + \frac{1}{2} \| \nabla \times g^\nu \|_{L^\infty}
  + \frac{1}{2} \| \nabla \times g^\nu \|_{L^\infty}
  \sup_{y} \| \varphi_{R,y} \omega^\nu \|^2_{L^2}.
\end{align*}

Integrating the above inequality with respect to time $t$, 
then taking the supremum over $y \in \Ro^2$ again:
\begin{align*}
  &\sup_y \| \varphi_{R,y} \omega^\nu \|^2_{L^2} + 
  \nu \sup_y \int^t_0 \| \nabla (\varphi_{R,y} \omega^\nu) \|^2_{L^2} \,ds\\ 
  \leq &\sup_y \| \varphi_{R,y} \omega^\nu_0 \|^2_{L^2} 
  + C_3 \nu \int^t_0 \sup_y \| \varphi_{R,y} \omega^\nu \|^2_{L^2} \,ds 
  + C_4 \int^t_0 e^{C t} \sup_y \| \varphi_{R,y} \omega^\nu \|^2_{L^2} \,ds\\ 
  & + \int^t_0 \| \nabla \times g^\nu \|_{L^\infty} \,ds
  + \int^t_0 \| \nabla \times g^\nu \|_{L^\infty}
  \sup_{y} \| \varphi_{R,y} \omega^\nu \|^2_{L^2} \,ds.
\end{align*}

By Gronwall inequality, we conclude:
\begin{align*}
  \sup_y \| \varphi_{R,y} \omega^\nu \|^2_{L^2} + 
  \nu \sup_y \int^t_0 \| \nabla (\varphi_{R,y} \omega^\nu) \|^2_{L^2} \,ds \leq &
  C e^{ C e^{ct}}.
\end{align*}

For the second term on the right-hand of \eqref{eq:d_v}:
\begin{align*}
  - \int_{\Ro^2} \varphi_{R,y} \nabla \delta p \varphi_{R,y} \delta u \, dx
  & = \int_{\Ro^2} \nabla \cdot (\varphi_{R,y}^2 \delta u) \delta p \, dx \\
  & = 2 \int_{\Ro^2} (\varphi_{R,y} \delta u) (\nabla \cdot \varphi_{R,y}) \delta p \, dx \\
  & \leq 2 \| \nabla \varphi_{R,y} \|_{L^\infty} 
  \| \varphi_{R,y} \delta u \|_{L^2} \| \varphi_{2R,y} \delta p \|_{L^2}.
\end{align*}

We need to treat the pressure in two parts.
Let 
\[
  p_1= - \frac{1}{\Delta} \nabla \cdot \nabla \cdot (\varphi_{4R,y} u \otimes u),
\]
and
\[
  p_2= - \frac{1}{\Delta} \nabla \cdot \nabla \cdot ((1-\varphi_{4R,y}) u \otimes u),
\]
so $p = p_1 + p_2$.

According to the proposition \ref{prop:uloc_convolution},
for the convolution in $p_1$, we have:
\begin{align*}
  - \int_{\Ro^2} \varphi_{R,y} \nabla p_1 \varphi_{R,y} \delta u \, dx
  & = \int_{\Ro^2} \nabla \cdot (\varphi_{R,y}^2 \delta u)  p_1 \, dx \\
  & = 2 \int_{\Ro^2} (\varphi_{R,y} \delta u) (\nabla \cdot \varphi_{R,y}) p_1 \, dx \\
  & \leq 2 \| \nabla \varphi_{R,y} \|_{L^\infty} \| \varphi_{R,y} \delta u \|_{L^2} \| \varphi_{2R,y} p_1 \|_{L^2} \\
  & \leq C \| \nabla \varphi_{R,y} \|_{L^\infty} \| \varphi_{R,y} \delta u \|_{L^2} \| u \otimes u \|_{L^2_{\Uloc}}\\
  & \leq C \| \nabla \varphi_{R,y} \|_{L^\infty} \| \varphi_{R,y} \delta u \|_{L^2} \| u \|_{L^2_{\Uloc}} \| u \|_{L^\infty}.
\end{align*}

For $\delta p_1$, using the same method as $p_1$, we obtain:
\begin{align*}
  - \int_{\Ro^2} \varphi_{R,y} \nabla \delta p_1 \varphi_{R,y} \delta u \, dx
  \leq & \frac{C}{R} \| \varphi_{R,y} \delta u \|_{L^2} \| u^\nu \otimes u^\nu - u \otimes u\|_{L^2_{\Uloc}} \\
  \leq & \frac{C}{R} \| \varphi_{R,y} \delta u \|_{L^2} \| u^\nu \otimes \delta u + \delta u \otimes u\|_{L^2_{\Uloc}} \\
  \leq & \frac{C}{R} \| \varphi_{R,y} \delta u \|_{L^2} \| \delta u \|_{L^2_{\Uloc}} (\| u^\nu \|_{L^\infty} + \| u \|_{L^\infty}).
\end{align*}

For $p_2$ and $\delta p_2$, we apply annular decomposition 
to the region $I-\varphi_{4R,y}(z)$:
\begin{align*}
  | \varphi_{R,y} \nabla p_2 |
  \leq &\varphi_{R,y} C \int_{\Ro^2} \frac{1}{|x-z|^3} [(1-\varphi_{4R,y}(z)) |u \otimes u|] \, dz \\
  \leq & \sum_{n=2}^\infty \frac{C}{2^{3n} R^3} 2^{2n} \int_{A_n} |u \otimes u| \, dz \\
  \leq & C \| |u \otimes u| \|_{L^1_{\Uloc}},
\end{align*}
where we set(see Figure \ref{fig:annular decomposition}):
\[
  A_n = \{z \in \Ro^2 \mid |z-y| \geq 2^n R, |z-y| \leq 2^{n+1} R\}.
\]

\begin{figure}[ht]
  \centering
  \includegraphics[width=0.6\textwidth]{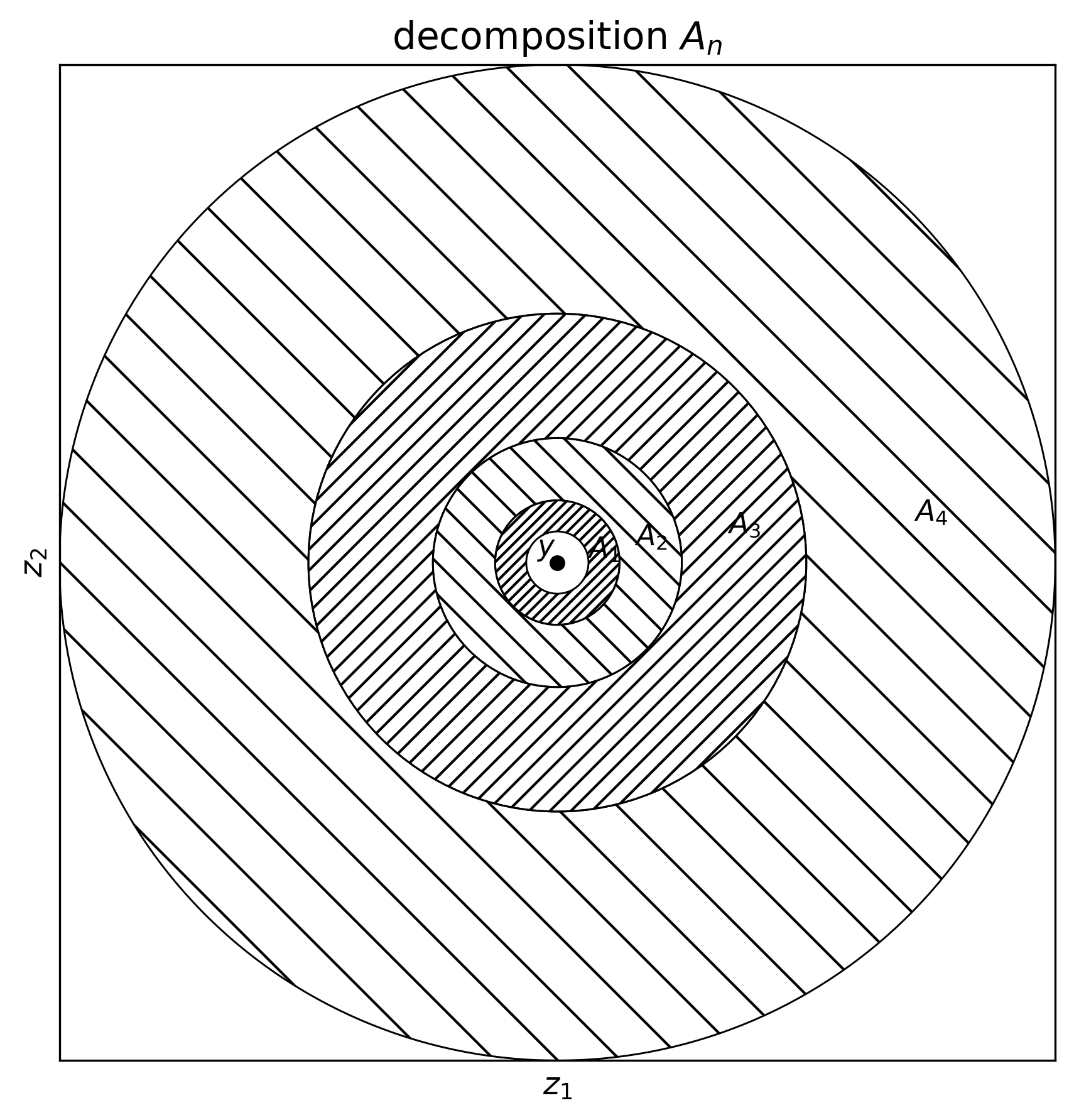}
  \caption{annular decomposition $A_n$}
  \label{fig:annular decomposition}
\end{figure}

Then, using the above inequality, it holds that:
\begin{align*}
  - \int_{\Ro^2} \varphi_{R,y} \nabla p_2 \varphi_{R,y} \delta u \, dx
  \leq & C \| |u \otimes u| \|_{L^1_{\Uloc}} \int_{\Ro^2} \varphi_{2R,y} \varphi_{R,y} \delta u \, dx\\
  \leq & C \| |u \otimes u| \|_{L^1_{\Uloc}} \| \varphi_{2R,y} \|_{L^2} \| \varphi_{R,y} \delta u \|_{L^2} \\
  \leq & C \| |u \otimes u| \|_{L^1_{\Uloc}} \| \varphi_{R,y} \delta u \|_{L^2}.
\end{align*}

For $\delta p_2$, we replace the $L^1_{\Uloc}$ norm:
\begin{align*}
  - \int_{\Ro^2} \varphi_{R,y} \nabla \delta p_2 \varphi_{R,y} \delta u \, dx
  \leq & C \| |u^\nu \otimes u^\nu - u \otimes u| \|_{L^1_{\Uloc}} \| \varphi_{R,y} \delta u \|_{L^2} \\
  \leq & C ( \| u^\nu \|_{L^\infty} + \| u \|_{L^\infty} ) \| \delta u \|_{L^2_{\Uloc}}
  \| \varphi_{R,y} \delta u \|_{L^2}.
\end{align*}

So we finally conclude:
\begin{align*}
  & - \int_{\Ro^2} \varphi_{R,y} \nabla \delta p \varphi_{R,y} \delta u \, dx \\
  \leq & \frac{C}{R} \| \varphi_{R,y} \delta u \|_{L^2} \| \delta u \|_{L^2_{\Uloc}} 
  (\| u^\nu \|_{L^\infty} + \| u \|_{L^\infty}) 
  + C ( \| u^\nu \|_{L^\infty} + \| u \|_{L^\infty} ) \| \delta u \|_{L^2_{\Uloc}}
  \| \varphi_{R,y} \delta u \|_{L^2} \\
  \leq & C' ( \| u^\nu \|_{L^\infty} + \| u \|_{L^\infty} ) \| \delta u \|_{L^2_{\Uloc}}
  \| \varphi_{R,y} \delta u \|_{L^2}.
\end{align*}

For the third term on the right-hand of \eqref{eq:d_v}, 
using the result in Section \ref{sec:est}, we have:
\begin{align*}
  & \int_{\Ro^2} (u^\nu \cdot \nabla \varphi_{R,y} ) \varphi_{2R,y} \delta u \varphi_{R,y} \delta u \,dx \\
  \leq & \| u^\nu \|_{L^\infty} \frac{1}{R} \| \nabla \varphi \|_{L^\infty} 
  \| \varphi_{2R,y} \delta u \|_{L^2} \| \varphi_{R,y} \delta u \|_{L^2} \\
  \leq & \frac{C}{R} \| \varphi_{R,y} \delta u \|_{L^2} \| \varphi_{2R,y} \delta u \|_{L^2} \| u^\nu \|_{L^\infty}\\
  \leq & \frac{C}{R} \| \varphi_{R,y} \delta u \|_{L^2}
  2^2 \sup_{y} \| \varphi_{R,y} \delta u \|_{L^2} \| u^\nu \|_{L^\infty}\\
  \leq & \frac{C}{R} \| \varphi_{R,y} \delta u \|_{L^2} \sup_{y} 
  \| \varphi_{R,y} \delta u \|_{L^2}  \| u^\nu \|_{L^\infty}.
\end{align*}

For the third term on the left-hand side of \eqref{eq:d_v},
we apply Hölder inequality($p>2$):
\begin{align*}
  \int_{\Ro^2} (\varphi_{R,y} \delta u) \cdot \nabla u (\varphi_{R,y} \delta u) \, dx
  & \leq \|\varphi_{2R,y} \nabla u\|_{L^p} 
  \|\varphi_{R,y} \delta u\|_{L^{\frac{2p}{p-2}}} \|\varphi_{R,y} \delta u\|_{L^2}.
\end{align*}

According to Biot-Savart law, it follows that:
\[
  u(x) = \frac{1}{2\pi} \int_{\Ro^2} 
  \frac{(x-z)^\perp}{|x-z|^2} \omega(z) dz.
\]

\begin{remark}\label{rem:biot-savart-remark}
  For the velocity $u$, we have:
  \begin{align*}
    \nabla u(x) = \frac{1}{2\pi} \int_{\Ro^2} \frac{1}{|x-y|^4}
    \begin{pmatrix}
      2(x_1-y_1)(x_2-y_2) & (x_2-y_2)^2 - (x_1-y_1)^2 \\
       (x_2-y_2)^2 - (x_1-y_1)^2 & -2(x_1-y_1)(x_2-y_2)
    \end{pmatrix}
    \omega(y) \, dy.
  \end{align*}
\end{remark}

For $\|\varphi_{2R,y} \nabla u\|_{L^p}$,
we find that all first-order partial derivatives of $u$ 
have the following form:
\[
  \frac{1}{2\pi} \int_{\Ro^2} \frac{C}{|x-z|^2} \omega(z) \, dz.
\]

For brevity, we exemplify this proof with $\frac{\partial u_1}{\partial x_1}$, 
and the others can be treated analogously.
\begin{align*}
  \varphi_{2R,y} \frac{\partial u_1}{\partial x_1} = & \varphi_{2R,y} \frac{1}{2\pi} \int_{\Ro^2} \frac{2(x_1-z_1)(x_2-z_2)}{|x-z|^4} \omega(z) \, dz\\
  = & \varphi_{2R,y} \frac{1}{2\pi} \int_{\Ro^2} \frac{2(x_1-z_1)(x_2-z_2)}{|x-z|^4} \omega(z) \varphi_{4R,y}(z) \, dz \\
  & + \varphi_{2R,y} \frac{1}{2\pi} \int_{\Ro^2} 
  \frac{2(x_1-z_1)(x_2-z_2)}{|x-z|^4} \omega(z) (1-\varphi_{4R,y}(z)) \, dz.
\end{align*}

By these two parts, we can also use annular decomposition:
\begin{align*}
  F_1 & \defeq \varphi_{2R,y} 
  \frac{1}{2\pi} \int_{\Ro^2} \frac{2(x_1-z_1)(x_2-z_2)}{|x-z|^4} \omega(z) \varphi_{4R,y}(z) \, dz,\\
  F_2 & \defeq \varphi_{2R,y} 
  \frac{1}{2\pi} \int_{\Ro^2} \frac{2(x_1-z_1)(x_2-z_2)}{|x-z|^4} \omega(z) (1-\varphi_{4R,y}(z)) \, dz.
\end{align*}

For $F_1$, we can use Biot-Savart law directly.
For $F_2$, it holds that:
\begin{align*}
  F_2 = & \varphi_{2R,y} \frac{1}{2\pi} \int_{\Ro^2} \frac{2(x_1-z_1)(x_2-z_2)}{|x-z|^4} \omega(z) (1-\varphi_{4R,y}(z)) \, dz \\
  = & \varphi_{2R,y} \frac{1}{2\pi} \int_{\Ro^2} \frac{2(x_1-z_1)(x_2-z_2)}{|x-z|^4} 
  (\frac{\partial u_2}{\partial x_1} - \frac{\partial u_1}{\partial x_2}) (1-\varphi_{4R,y}(z)) \, dz \\
  = & \varphi_{2R,y} \frac{1}{2\pi} \int_{\Ro^2} \frac{2(x_1-z_1)(x_2-z_2)}{|x-z|^4}
  \bigg[ \partial_1 [(1-\varphi_{4R,y}(z))u_2]- [\partial_1(1-\varphi_{4R,y}(z))u_2]  \\
  & \qquad \qquad \quad -\partial_2 [(1-\varphi_{4R,y}(z))u_1] + [\partial_2(1-\varphi_{4R,y}(z))u_1] \bigg] \, dz.
\end{align*}

Of the four terms within the square brackets, the first and third belong to one type, 
and the second and fourth belong to another.

For the first type in $F_2$, using the $A_n$ from the annular decomposition:
\begin{align*}
  & \big| \varphi_{2R,y} \frac{1}{2\pi} \int_{\Ro^2} \frac{2(x_1-z_1)(x_2-z_2)}{|x-z|^4}
  \big[ \partial_1 [(1-\varphi_{4R,y}(z))u_2] \big] \, dz \big| \\
  \leq & \varphi_{2R,y} C \int_{\Ro^2} \frac{1}{|x-z|^3} [(1-\varphi_{4R,y}(z)) |u|] \, dz \\
  \leq & \sum_{n=2}^\infty \frac{C}{2^{3n} R^3} 2^{2n} \int_{A_n} |u| \, dz \leq C \| u \|_{L^1_{\Uloc}} \leq C \| u \|_{L^\infty}.
\end{align*}

For the second type in $F_2$, we find that $\partial_1 (1-\varphi_{4R,y}(z))$
is non-zero and bounded only under the condition $|z| \in [4R,8R]$.
So we have:
\begin{align*}
  & \bigg| \varphi_{2R,y} \frac{1}{2\pi} \int_{\Ro^2} \frac{2(x_1-z_1)(x_2-z_2)}{|x-z|^4}
  \big[ \partial_1 [(1-\varphi_{4R,y}(z))u_2] \big] \, dz \bigg| \\
  \leq & \varphi_{2R,y} C \int_{\Ro^2} \frac{1}{|x-z|^2} 
  [\partial_1 (1-\varphi_{4R,y}(z)) \varphi_{8R,y} u_2] \, dz \\
  \leq & C \| u \|_{L^\infty}.
\end{align*}

Consequently, we obtain:
\begin{align*}
  & \int_{\Ro^2} (\varphi_{R,y} \delta u) \cdot \nabla u (\varphi_{R,y} \delta u) \, dx \\
  \leq & C p \|\varphi_{2R,y} \omega \|_{L^p} 
  \|\varphi_{R,y} \delta u\|_{L^{\frac{2p}{p-2}}} \|\varphi_{R,y} \delta u\|_{L^2} 
  + C \| u \|_{L^\infty} \|\varphi_{R,y} \delta u\|^2_{L^2}.
\end{align*}

For the first term in the right side of the above equality, we have:
\begin{align*}
  & C p \|\varphi_{2R,y} \omega \|_{L^p} 
  \|\varphi_{R,y} \delta u\|_{L^{\frac{2p}{p-2}}} \|\varphi_{R,y} \delta u\|_{L^2} \\
  \leq & C \| \omega_0 \|_{L^\infty} p 
  \|\varphi_{R,y} \delta u\|^{1-\frac{2}{p}}_{L^{2}} \|\varphi_{R,y} \delta u\|^{\frac{2}{p}}_{L^{\infty}}
  \|\varphi_{R,y} \delta u\|_{L^2}.
\end{align*}

\begin{remark}\label{rem:inf_f}
Consider the function \( f(p) = p a^{-\frac{2}{p}} \)
for \( p > 2 \) and \( a > 0 \). we have:
\[
  f'(p) = a^{-\frac{2}{p}} \left( 1 + \frac{2 \ln a}{p} \right).
\]

It holds that:
\[
  \inf_{p > 2} f(p) = 
  \begin{cases} 
    2 a^{-1}, & \text{if } a \geq 1, \\
    -2 e \ln a, & \text{if } 0 < a < 1.
  \end{cases}
\]
\end{remark}

For $p \|\varphi_{R,y} \delta u\|^{-\frac{2}{p}}_{L^{2}}$,
we use the remark \ref{rem:inf_f}.
When $\|\varphi_{R,y} \delta u\|_{L^{2}} > 1$, we replace this term with $2$.

Moreover, it is important to note the case $\|\varphi_{R,y} \delta u(t)\|_{L^{2}} =0$ in some times $t$.

(For $\forall \varepsilon > 0$,
the time derivative allows us to replace $\|\varphi_{R,y} \delta u\|_{L^{2}}$ with
$(\|\varphi_{R,y} \delta u\|_{L^{2}} + \varepsilon)$.)

By Young inequality, it follows that:
\begin{align*}
  & C p \|\varphi_{2R,y} \omega \|_{L^p} 
  \|\varphi_{R,y} \delta u\|_{L^{\frac{2p}{p-2}}} \|\varphi_{R,y} \delta u\|_{L^2} \\
  \leq & C \| \omega_0 \|_{L^\infty} 
  \|\varphi_{R,y} \delta u\|^{\frac{2}{p}}_{L^{\infty}}\|\varphi_{R,y} \delta u\|^2_{L^2} 
  \max \{-e \ln (\|\varphi_{R,y} \delta u\|_{L^{2}} + \varepsilon), 1\} \\
  \leq & C \| \omega_0 \|_{L^\infty}
  (1+\| \varphi_{R,y} \delta u\|_{L^{\infty}})
  \| \varphi_{R,y} \delta u\|^2_{L^2} 
  \max \{-e \ln (\|\varphi_{R,y} \delta u\|_{L^{2}} + \varepsilon), 1\}\\
  \leq & C \| \omega_0 \|_{L^\infty}
  (1+\|u\|_{L^\infty}+\|u^\nu\|_{L^\infty})
  \| \varphi_{R,y} \delta u\|^2_{L^2} 
  \max \{-e \ln (\|\varphi_{R,y} \delta u\|_{L^{2}} + \varepsilon), 1\}.
\end{align*}

Recalling the equation for $\varphi_{R,y} \delta u$, we derive:
\begin{align*}
  & \frac{1}{2} \frac{d}{dt} \|\varphi_{R,y} \delta u\|^2_{L^2} \\
  = & -\int_{\Ro^2} (\varphi_{R,y} \delta u) \cdot \nabla u (\varphi_{R,y} \delta u) \, dx
  + \int_{\Ro^2} (u^\nu \cdot \nabla \varphi_{R,y} ) \varphi_{2R,y} \delta u \varphi_{R,y} \delta u \,dx  \\
  & + \nu \int_{\Ro^2} \varphi_{R,y} \Delta u^\nu  \varphi_{R,y} \delta u \,dx
  - \int_{\Ro^2} \varphi_{R,y} \nabla \delta p \varphi_{R,y} \delta u \, dx
  + \int_{\Ro^2} \varphi_{R,y} \delta g \varphi_{R,y} \delta u \, dx.
\end{align*}

By the all of the above inequalities and Remark \ref{rem:inf_f}, we have:
\begin{align*}
  & \frac{d}{dt} (\|\varphi_{R,y} \delta u\|_{L^{2}} + \varepsilon) \\
  \leq & C \| \omega_0 \|_{L^\infty} (1+\|u\|_{L^\infty}+\|u^\nu\|_{L^\infty})
  (\|\varphi_{R,y} \delta u\|_{L^{2}} + \varepsilon) 
  \max \{-e \ln (\|\varphi_{R,y} \delta u\|_{L^{2}} + \varepsilon), 1\} \\
  & + C \| u \|_{L^\infty} (\|\varphi_{R,y} \delta u\|_{L^{2}} + \varepsilon)
  + C \sup_{y} (\|\varphi_{R,y} \delta u\|_{L^{2}} + \varepsilon) 
  \| u^\nu \|_{L^\infty}\\
  & + \nu \|  \varphi_{R,y} \nabla \omega^\nu \|_{L^2} 
  + C \| \delta u \|_{L^2_{\Uloc}} (\| u \|_{L^\infty} + \| u^\nu \|_{L^\infty})
  + \| \delta g \|_{L^2_{\Uloc}}.
\end{align*}

(dihuan)
Finally, by the time integration of the inequality 
and then taking the supremum over $y$, we conclude:
\begin{align*}
  \| \delta u \|_{L^2_{\Uloc}} & \leq 
  C ( \nu t )^{\frac{1}{2}} e^{C e^{ct}}
  + ( \| \delta u_0 \|_{L^2_{\Uloc}} + \| \delta g \|_{L^1([0,t];L^2_{\Uloc})})
  \bigg[C e^{ C \exp(C t)} + C e^{\exp({C \exp(C t)})}\bigg] \\
  & \leq C ( \nu t )^{\frac{1}{2}} e^{C e^{ct}}
  + C ( \| \delta u_0 \|_{L^\infty} + \| \delta g \|_{L^1([0,t];L^\infty)})
  \bigg[C e^{ C \exp(C t)} + C e^{\exp({C \exp(C t)})}\bigg]\\
  & \leq C ( \nu t )^{\frac{1}{2}} e^{C e^{ct}}
  + C ( \| \delta u_0 \|_{L^\infty} + \| \delta g \|_{L^1([0,t];L^\infty)})
   \bigg[ C e^{\exp({C \exp(C t)})}\bigg].
\end{align*}
when $\varepsilon \rightarrow 0$.

When $\nu \rightarrow 0$, it holds that $\sup_{t \in [0,T]} \| \delta u \|_{L^2_{\Uloc}} \rightarrow 0$.
By the boundedness of the vorticity and proposition \ref{prop:inequality_uloc_l2}, 
it follows that:
\begin{align*}
  \sup_{t \in [0,T]} \| \delta u \|_{L^\infty}
  \leq C \sup_{t \in [0,T]} \| \delta u \|^{\frac{1}{2}}_{L^2_{\Uloc}} 
  \| \delta \omega \|^{\frac{1}{2}}_{L^\infty} \rightarrow 0.
\end{align*}

\subsection{Proof of Theorem \ref{thm:convergence_omega}}
In this subsection, we employ the approximation equations 
to prove the Theorem \ref{thm:convergence_omega}.
There are two types of approximation equations.
The first type is:
\begin{equation}\label{eq:approx_omega_a}\tag{4.2a}
  \left.\begin{array}{r}
    \partial_t \omega^\nu_n + u^\nu \cdot \nabla \omega^\nu_n - 
    \nu \Delta \omega^\nu_n = \nabla \times g^\nu \\
    \omega^\nu_n \big|_{t=0} = P_n \omega_0^\nu
  \end{array}\right\},
\end{equation}

\begin{equation}\label{eq:approx_omega_b}\tag{4.2b}
  \left.\begin{array}{r}
    \partial_t \omega_n + u \cdot \nabla \omega_n = \nabla \times g \\
    \omega_n \big|_{t=0} = P_n \omega_0
  \end{array}\right\}.
\end{equation}

They are respectively the approximation equations 
of N-S and Euler equations.

And the second type is the approximation to the first type:
\begin{equation}\label{eq:approx_app_a}\tag{4.3a}
  \left.\begin{array}{r}
    \partial_t \omega^\nu_{n,m} + (S_m u^\nu) \cdot \nabla \omega^\nu_{n,m} - 
    \nu \Delta \omega^\nu_{n,m} = \nabla \times g^\nu \\
    \omega^\nu_{n,m} \big|_{t=0} = P_n \omega_0^\nu
  \end{array}\right\},
\end{equation}

\begin{equation}\label{eq:approx_app_b}\tag{4.3b}
  \left.\begin{array}{r}
    \partial_t \omega_{n,m} + (S_m u) \cdot \nabla \omega_{n,m} = \nabla \times g \\
    \omega_{n,m} \big|_{t=0} = P_n \omega_0
  \end{array}\right\}.
\end{equation}

Note that the initial data are the same for both types of equations.

Let $\delta \omega = \omega^\nu - \omega$.
Applying the triangular inequality with equations above, 
we can get the conclusion of $\|\delta \omega\|_{L^2_{\Uloc}}$.
So we set:
\begin{align*}
  & E_1=\omega^\nu_n - \omega^\nu, \quad E_2=\omega_m - \omega_n, \quad
  E_3=\omega_{n,m} - \omega_n, \\
  & E_4=\omega^\nu_{n,m} - \omega^\nu_n, \quad E_5=\omega^\nu_{n,m} - \omega_{n,m}.
\end{align*}

For $E_1$ and Equation \eqref{eq:approx_omega_a}, we can write:
\begin{align*}
  & \partial_t \varphi_{R,y}(\omega^\nu_n - \omega^\nu) + 
  u^\nu \cdot \nabla [\varphi_{R,y}(\omega^\nu_n - \omega^\nu)] \\ 
  = & \nu \varphi_{R,y} \Delta (\omega^\nu_n - \omega^\nu) +
  (\omega^\nu_n - \omega^\nu) u^\nu \cdot \nabla \varphi_{R,y}.
\end{align*}

And we have:
\begin{align*}
  \frac{1}{2} \frac{d}{dt} \| \varphi_{R,y} E_1 \|^2_{L^2}
  \leq \nu \int_{\Ro^2} \varphi_{R,y} \Delta E_1 \varphi_{R,y} E_1 \, dx  
  + \int_{\Ro^2} (E_1) u^\nu \cdot (\nabla \varphi_{R,y}) \varphi_{R,y} E_1 \, dx.
\end{align*}

Applying the integration by parts and Young inequality, we derive:
\begin{align*}
  & \frac{d}{dt} \| \varphi_{R,y} E_1 \|^2_{L^2} + 
  2 \nu \| \varphi_{R,y} \nabla E_1 \|^2_{L^2} \\
  \leq & 4 \nu \int_{\Ro^2} (\nabla \varphi_{R,y}) E_1 (\varphi_{R,y} \nabla E_1) \, dx + 
  2 \int_{\Ro^2} (E_1) u^\nu \cdot (\nabla \varphi_{R,y}) \varphi_{R,y} E_1 \, dx \\
  \leq & C \nu \sup_y \| \varphi_{R,y} E_1 \|_{L^2} \| \varphi_{R,y} \nabla E_1 \|_{L^2} + C \| \varphi_{R,y} E_1 \|_{L^2} \sup_y \| \varphi_{R,y} E_1 \|_{L^2} \| u^\nu \|_{L^\infty} \\
  \leq & C_2 \nu (\sup_y \| \varphi_{R,y} E_1 \|_{L^2})^2 
  + \nu \| \varphi_{R,y} \nabla E_1 \|^2_{L^2}
  + C \| \varphi_{R,y} E_1 \|_{L^2} \sup_y \| \varphi_{R,y} E_1 \|_{L^2} \| u^\nu \|_{L^\infty}\\
  \leq & C_2 \nu (\sup_y \| \varphi_{R,y} E_1 \|_{L^2})^2 
  + \nu \| \varphi_{R,y} \nabla E_1 \|^2_{L^2}
  + C \| \varphi_{R,y} E_1 \|_{L^2} \sup_y \| \varphi_{R,y} E_1 \|_{L^2} e^{C t}.
\end{align*}

By Gronwall inequality, it follows that:
\begin{align*}
  \sup_{y} \| \varphi_{R,y} E_1 \|^2_{L^2} 
  + \sup_{y} \nu \int_{0}^{t}  \| \varphi_{R,y} \nabla E_1 \|^2_{L^2} \,ds
  \leq & \sup_{y} \| \varphi_{R,y} E_1(0) \|^2_{L^2} (e^{C \nu t} + e^{C \exp(C t)}).
\end{align*}

Considering the initial data, we find that
$\sup_{y} \| \varphi_{R,y} E_1(0) \|_{L^2} \rightarrow 0 (n \rightarrow \infty)$.

So we have 
\begin{align*}
  \sup_{t \in [0,T]} \sup_{y} \| \varphi_{R,y} (\omega^\nu_n - \omega^\nu) \|_{L^2} \rightarrow 0,
\end{align*}
when $n \rightarrow \infty$.

For $E_2$, we employ Fatou lemma:
\[
  \| \varphi_{R,y}(\omega - \omega_n) \|_{L^2} \leq 
  \liminf_{m \to \infty} \| \varphi_{R,y}(\omega_m - \omega_n) \|_{L^2}
  =\liminf_{m \to \infty} \| \varphi_{R,y} E_2 \|_{L^2}.
\]

By the equation \eqref{eq:approx_omega_b} and the definition of $E_2$, we have:
\begin{align*}
  \partial_t \varphi_{R,y} E_2 + u \cdot \nabla (\varphi_{R,y} E_2)
  = E_2 (u \cdot \nabla \varphi_{R,y}).
\end{align*}

Integrating the equation, it holds that:
\begin{align*}
  \frac{1}{2} \frac{d}{dt} \| \varphi_{R,y} E_2 \|^2_{L^2}
  & \leq C \| \varphi_{R,y} E_2 \|_{L^2} \sup_y \| \varphi_{R,y} E_2 \|_{L^2} \| u \|_{L^\infty} \\
  & \leq C \| \varphi_{R,y} E_2 \|_{L^2} \sup_y \| \varphi_{R,y} E_2 \|_{L^2} e^{C t}.
\end{align*}

We next integrate the time derivative and take the supremum over $y$:
\begin{align*}
  \sup_{y}\| \varphi_{R,y} E_2 \|_{L^2} 
  \leq \sup_{y} \| \varphi_{R,y} E_2(0) \|_{L^2} e^{C \exp(C t)}.
\end{align*}

Thus, when $n \rightarrow \infty$,
we have
\begin{align*}
  \sup_{t \in [0,T]} \sup_{y} \| \varphi_{R,y} (\omega - \omega_n) \|_{L^2} \rightarrow 0.
\end{align*}

For Equation \eqref{eq:approx_omega_b}, \eqref{eq:approx_app_b} and $E_3$, we can write:
\begin{align*}
  \partial_t \varphi_{R,y} E_3 + u \cdot \nabla (\varphi_{R,y} E_3)
  = E_3 (u \cdot \nabla \varphi_{R,y}) - \varphi_{R,y} (P_{\geq m} u) \cdot \nabla \omega_{n,m}.
\end{align*}

After integrating the equation, we need the control of $\| \nabla \omega_{n,m} \|_{L^2}$:
\begin{align*}
  &\frac{1}{2} \frac{d}{dt} \| \varphi_{R,y} E_3 \|^2_{L^2} \\
  \leq & C \| \varphi_{R,y} E_3 \|_{L^2} \sup_y \| \varphi_{R,y} E_3 \|_{L^2} 
  \| u \|_{L^\infty} + C 2^{-m} \| \varphi_{R,y} E_3 \|_{L^2}  \| \nabla \omega_{n,m} \|_{L^2}.
\end{align*}

According to the approximation equation in \eqref{eq:approx_app_b},
we multiply it by $\Delta \omega_{n,m}$ and integrate the equation:
\begin{align*}
  \frac{1}{2} \frac{d}{dt} \| \nabla \omega_{n,m} \|^2_{L^2} 
  & \leq C \| \nabla P_m u \|_{L^\infty} \| \nabla \omega_{n,m} \|^2_{L^2} \\
  & \leq C \| \Sigma^{m-1}_{q \geq 0} \nabla P_q u \|_{L^\infty} \| \nabla \omega_{n,m} \|^2_{L^2} \\
  & \leq C \| \sqrt{m} \|\nabla P_q u \|_{\ell^2} \|_{L^\infty} \| \nabla \omega_{n,m} \|^2_{L^2} \\
  & \leq C \sqrt{m} \| \nabla u \|_{\dot F^0_{\infty,2}}
  \| \nabla \omega_{n,m} \|^2_{L^2}.
\end{align*}

By the boundedness of Riesz transform in BMO space
and the equivalence between BMO norm
and $\dot F^0_{\infty,2}$ norm(see Remark \ref{rem:bmo-fsp-equivalence}
and \ref{rem:riesz-bmo-properties}),
we have $ \| \nabla u \|_{\dot F^0_{\infty,2}} < \infty $.

So we obtain:
\begin{align*}
  \| \nabla \omega_{n,m} \|^2_{L^2} & \leq
  \| \nabla P_n \omega_0 \|^2_{L^2} e^{C \sqrt{m} t},
\end{align*}
and
\begin{align*}
  \frac{d}{dt} \| \varphi_{R,y} E_3 \|_{L^2} 
  \leq  C \sup_y \| \varphi_{R,y} E_3 \|_{L^2}  
  e^{C t} +  C 2^{-m} \| \nabla P_n \omega_0 \|_{L^2} e^{C \sqrt{m} t}.
\end{align*}

Applying the proposition of operator $P_n$, 
it follows that$ \| \nabla P_n \omega_0 \|_{L^2} < \infty $ and
\begin{align*}
  \sup_{y} \| \varphi_{R,y} E_3 \|_{L^2} 
  \leq & \sup_{y} \| \varphi_{R,y} E_3(0) \|_{L^2} e^{C \exp(C t)}
  + C 2^{-m} \frac{2}{C \sqrt{m}} (e^{C \sqrt{m} t/2}-1).
\end{align*}

For any fixed $t$, the second term on the right-hand side 
can be made arbitrarily small by choosing $m$ sufficiently large.
So we conclude that:
\begin{align*}
  \sup_{t \in [0,T]} \sup_{y} \| \varphi_{R,y} (\omega_{n,m} - \omega_n) \|_{L^2} \rightarrow 0,
\end{align*}
when $m \rightarrow \infty$.

In the same way, we have the conclusion of $E_4$ and Equation \eqref{eq:approx_app_a}:
\begin{align*}
  & \frac{d}{dt} \| \varphi_{R,y} E_4 \|^2_{L^2} + 
  2 \nu \| \varphi_{R,y} \nabla E_4 \|^2_{L^2} \\
  \leq & 4 \nu \int_{\Ro^2} (\nabla \varphi_{R,y}) E_4 (\varphi_{R,y} \nabla E_4) \, dx 
  + C \| \varphi_{R,y} E_4 \|_{L^2} \sup_y \| \varphi_{R,y} E_4 \|_{L^2} \| u^\nu \|_{L^\infty} \\
  & +  C 2^{-m} \| \varphi_{R,y} E_4 \|_{L^2} \| \nabla \omega^\nu_{n,m} \|_{L^2}\\
  \leq & C \nu (\sup_y \| \varphi_{R,y} E_4 \|_{L^2})^2 
  + \nu \| \varphi_{R,y} \nabla E_4 \|^2_{L^2} 
  + C \| \varphi_{R,y} E_4 \|_{L^2} \sup_y \| \varphi_{R,y} E_4 \|_{L^2} e^{C t} \\
  & +  C 2^{-m} \| \varphi_{R,y} E_4 \|_{L^2} \| \nabla P_n \omega^\nu_0 \|_{L^2} e^{C \sqrt{m} t}.
\end{align*}

Then, we obtain that:
\begin{align*}
  &\sup_{y} \| \varphi_{R,y} E_4 \|^2_{L^2} 
  + \sup_{y} \nu \int_{0}^{t} \| \varphi_{R,y} \nabla E_4 \|^2_{L^2} \,ds \\
  \leq & \sup_{y} \| \varphi_{R,y} E_4(0) \|^2_{L^2} e^{C \exp(C t)}
  + C 2^{-m} \frac{2}{C \sqrt{m}} (e^{C \sqrt{m} t/2}-1).
\end{align*}

So when $m \rightarrow \infty$,
we derive
\begin{align*}
  \sup_{t \in [0,T]} \sup_{y} \| \varphi_{R,y} (\omega^\nu_{n,m} - \omega^\nu_n) \|_{L^2} \rightarrow 0.
\end{align*}

For Equation \eqref{eq:approx_app_a}, \eqref{eq:approx_app_b} and $E_5$,
we obtain the following equation:
\begin{align*}
  \partial_t E_5 + S_m u \cdot \nabla E_5 + S_m \delta u \cdot \nabla \omega_{n,m}
  =  \nu \Delta \omega^\nu_{n,m} + \nabla \times \delta g,
\end{align*}
which employ
\begin{align*}
  & \partial_t \varphi_{R,y} E_5 + S_m u \cdot \nabla (\varphi_{R,y} E_5) 
  + \varphi_{R,y} S_m \delta u \cdot \nabla \omega_{n,m} \\
  = & \nu \varphi_{R,y} \Delta \omega^\nu_{n,m} 
  + E_5 (S_m u \cdot \nabla \varphi_{R,y}) + \varphi_{R,y} \nabla \times \delta g.
\end{align*}

Multiplying the above equation by $\varphi_{R,y} E_5$ and then integrating it, we have:
\begin{align*}
  & \frac{1}{2} \frac{d}{dt} \| \varphi_{R,y} E_5 \|^2_{L^2} \\
  \leq & \| \varphi_{R,y} S_m \delta u \|_{L^\infty}
  \| \nabla \omega_{n,m} \|_{L^2} \| \varphi_{R,y} E_5 \|_{L^2}
  + C \nu \| \Delta \omega^\nu_{n,m} \|_{L^2} \| \varphi_{R,y} E_5 \|_{L^2} \\
  & + C \| \varphi_{R,y} E_5 \|_{L^2} \sup_y \| \varphi_{R,y} E_5 \|_{L^2} 
  \| S_m u \|_{L^\infty} + \| \varphi_{R,y} \nabla \times \delta g \|_{L^2}
  \| \varphi_{R,y} E_5 \|_{L^2}.
\end{align*}

According to the approximation equation \eqref{eq:approx_app_a},
we also have:
\begin{align*}
  \nu \| \Delta \omega^\nu_{n,m} \|^2_{L^2} 
  \leq \| \nabla P_n \omega^\nu_0 \|^2_{L^2} e^{C \sqrt{m} t}.
\end{align*}

We conclude that:
\begin{align*}
  & \sup_y \| \varphi_{R,y} E_5 \|_{L^2} \\
  \leq & \bigg[ \int_0^t \| \delta u(s) \|_{L^\infty} \, ds +C \nu
  + \sup_y \| \varphi_{R,y} E_5(0) \|_{L^2} 
  + \| \nabla \times \delta g \|_{L^1([0,t]; L^2_{\Uloc})} \bigg] e^{C \exp(C t)}\\
  \leq & \bigg[ \int_0^t \| \delta u(s) \|_{L^\infty} \, ds +C \nu
  + \sup_y \| \varphi_{R,y} E_5(0) \|_{L^2} 
  + \| \nabla \times \delta g \|_{L^1([0,t]; L^\infty)} \bigg] e^{C \exp(C t)}.
\end{align*}

We find when $m \rightarrow \infty$ and $\nu \rightarrow 0$,
\begin{align*}
  \sup_{t \in [0,T]} \sup_y \| \varphi_{R,y} (\omega^\nu_{n,m} - \omega_{n,m}) \|_{L^2} \rightarrow 0.
\end{align*}

Finally, applying the triangular inequality,
we have:
\begin{align*}
  \sup_{t \in [0,T]}\|\delta \omega\|_{L^2_{\Uloc}} \leq 
  \sum_{k=1}^{5} \sup_{t \in [0,T]} \sup_{y} \| \varphi_{R,y} E_k \|_{L^2} \rightarrow 0,
\end{align*}
when $\nu \rightarrow 0$ and $n,m \rightarrow \infty$.

We finish the proof of Theorem \ref{thm:convergence_omega}.
\qed

\section*{Acknowledgments}
H. Tai and X. Zheng were partially supported by the National Natural Science Foundation of China under
grant No. 12371231.

\end{document}